\documentclass[12pt,a4paper]{amsart}

\usepackage{amsmath,amssymb, amsfonts}

\usepackage{bbold}

\newtheoremstyle{main}% name
  {12pt}%      Space above, empty = `usual value'
  {12pt}%      Space below
  {\slshape}% Body font
  {}%         Indent amount (empty = no indent,\parindent  = para inden
  {\scshape}% Thm head font
  {.}%        Punctuation after thm head
  {12pt}% Space after thm head: \newline = linebreak
  {}%         Thm head spec

\theoremstyle{main}

\newtheorem{thm}{Theorem}%[subsection]

\newtheorem{cor}{Corollary}%[section][thm]
\newtheorem{lem}{Lemma}%[section][thm]

\newtheoremstyle{Defn}
 {}%      Space above, empty = `usual value'
  {12pt}%      Space below
  {\normalfont}% Body font
 {}%         Indent amount (empty = no indent,\parindent  = para indent)
  {\scshape}% Thm head font
  {.}%        Punctuation after thm head
  {12pt}% Space after thm head: \newline = linebreak
  {}%         Thm head spec

\theoremstyle{Defn}
\newtheorem{defn}{Definition}%[section]
\newtheorem{problem}{Problem}%[section]
\newtheorem{rem}{Remark}%[subsection]
\newtheoremstyle{Defn}
 {}%      Space above, empty = `usual value'
  {12pt}%      Space below
  {\normalfont}% Body font
  {}%         Indent amount (empty = no indent,\parindent  = para indent)
  {\scshape}% Thm head font
  {.}%        Punctuation after thm head
  {12pt}% Space after thm head: \newline = linebreak
  {}%         Thm head spec

\title [Polyhedra for which a homotopy domination
 is an equivalence]
{Polyhedra for which
 every homotopy domination over itself is a homotopy
 equivalence }

\author [D. Ko{\l}odziejczyk] {Danuta Ko{\l}odziejczyk}

\address{Faculty of Mathematics and Information Science,
Warsaw
University of Techno\-lo\-gy, ul. Koszykowa 75,
00-662 Warsaw,
Poland}

\email{dkolodz@mimuw.edu.pl; dakolodz@gmail.com}



\thanks{Research partially supported by the Ministry of
Science and Higher Education grant \# 1 P03A 005 30.
}

\begin{document}

\bibliographystyle{alpha}

\maketitle

%\today

\begin{abstract}
We consider a natural question: {\it "Is it true that each
homotopy domi\-nation of a polyhedron over itself is
a homotopy equi\-va\-lence?"\/} and a strongly related
problem of K. Borsuk (1967): {\it "Is it true that two ANR's
homotopy dominating each other have the same homotopy type?"\/}
The answer was earlier known to be positive for
manifolds (Bernstein-Ganea, 1959), $1$-dimensional polyhedra
and polyhedra with polycyclic-by-finite fundamental groups
(DK, 2005). Thus one may ask, if there exists a counterexample
among $2$-dimensional polyhedra with so\-lu\-ble fundamental
groups. In this paper we show that it cannot be found in the
class of $2$-dimensional polyhedra  with so\-lu\-ble
fundamental groups $G$
 with cd$G \leq 2$ (and soluble can be replaced
 here by a wider class of elementary amenable groups).
We prove more general fact, that there are no counterexamples
in the class of $2$-dimensional polyhedra, whose fundamental 
groups have finite aspherical presentations and are Hopfian
(or more general, weakly Hopfian).
In particular, a counterexample does not
exist also among $2$-dimensional polyhedra, whose fundamental
groups are knot groups and in the class of $2$-dimensional 
polyhedra with one-related torsion-free, Hopfian fundamental 
groups.
The results can be applied also, for example, to hyperbolic
groups or limit groups with finite aspherical presentations.

For the same classes of polyhedra we get a positive
answer to another open question: {\it "Are the homotopy
types of two quasi-homeomorphic ANR's equal?"\/}
\end{abstract}

%\maketitle

\section{Introduction}
In this paper we study two natural but still open problems: 
{\it "Is
it true that for every polyhedron $P$, each homotopy domination of
$P$ over itself is a homotopy equivalence?"}\/  
and the famous problem of K. Borsuk
 (1967) [B1, Ch.IX, Problem (12.7)]: {\it "Is it true that two
 $ANR's$ homotopy
dominating each other have the same homotopy type?"\/}
(By a polyhedron we mean, as usual, a
finite one. For convenience, we will assume
without loss of generality, that each
polyhedron and $ANR$ is connected.)

\smallskip
They are closely related to another open problem in
geometric topology:
{\it "Are the homotopy types (or equivalently, shapes) of two
quasi-homeomorphic ANR's equal?"\/} [B2, Problem (12.7), p.
233] that will be also considered here.

\smallskip
In dimension $1$ the answers to all the above questions are
positive. Indeed, every $1$-dimensional polyhedron has the
homotopy type of a finite wedge of circles $S^1$, so $K(F,1)$,
where $F$ is a free, finitely generated group. It is
well-known that every finitely generated, free group is
Hopfian (Nielsen 1921, Hopf 1931).

\smallskip
By the results of  I. Bernstein and T. Ganea (1959),
 if $P$ is a manifold, then every homotopy
domination of $P$ over itself is a homotopy equivalence
$[$BG$]$. (It was generalized to the so-called Poincare
complexes in [Kw]).

\smallskip
For polyhedra with polycyclic-by-finite fundamental groups the
answers are also posi\-tive (see [K, Theorem 3,
 Theorem 5]). Thus, one may
ask how about $2$-dimensional polyhedra, in particular
$2$-dimensional polyhedra with so\-lu\-ble
fundamental groups.

\smallskip
In this paper we prove that for polyhedra
$P$ with $\dim P = 2$ and soluble (or, more general,
elementary amenable) fundamental groups $G$ satisfying
cd$G \leq 2$, there are no counterexamples. This is
a corollary to the main result that there are
no counterexamples among $2$-dimensional polyhedra, 
whose fundamental groups have
 finite aspherical presentations and
are weakly Hopfian (see Definition 3). It should be noted
that at
present there is not known any example of
a fini\-tely presented group which is not weakly Hopfian.

\smallskip
We also consider some other classes of finitely presented
groups satisfying conditions of our main theorem.
As one of the corollaries, we obtain that for each
$2$-dimensional polyhedron whose fundamental group is
a knot group, every homotopy domination over itself is
a homotopy equivalence. The same we get for $2$-dimensional
polyhedra whose fundamental groups are one-related,
torsion-free and Hopfian (note that in many cases
one-related, torsion-free groups are known to be Hopfian,
see, for example, [SSp]). We also obtain positive results
for $2$-dimensional polyhedra whose fundamental groups
are hyperbolic groups or limit groups with finite
aspherical presentations.

\smallskip
In a consequence, we also answer positively the third
question for the same classes of $2$-dimensional polyhedra.

\medskip
The results of this paper were presented by the author
at the "2010 International Conference on Topology and its
Applications" (Nafpaktos), the "25th Summer Conference on
Topology and its Applications 2010" (Kielce),
and were included in some bigger conference lectures.

\smallskip
\section{Preliminaries}

\smallskip
\begin{defn} Let  $\mathcal{S}$ be a class of groups.
A group $G$ is called {\it poly-$\mathcal{S}$\/} if it has
a finite series  $G = G_0 \triangleright
G_1 \triangleright \ldots \rhd G_l = 1$, for which each factor
$G_{i-1}/G_{i}$ $\in \mathcal{S}$ (where $1 \leq i \leq l$).
\end{defn}

\begin{defn} Let $\mathcal{P}, \mathcal{S}$ be some
classes of groups.  A group $G$ is said to be $\mathcal{P}$-{\it
by\/}-$\mathcal{S}$ if it has a normal subgroup
%$A \lhd G$,
$A \in \mathcal{P}$ such that $G/A \in \mathcal{S}$.
\end{defn}

\begin{defn}
 (i) A group $G$ is {\it
Hopfian\/} if every epimorphism $f:G \rightarrow G$ is an
automorphism (equivalently, $N = 1$ is the only normal subgroup
for which $G/N \cong G$). (ii) A group $G$ is {\it weakly
Hopfian\/} if $G = K \rtimes H$ and $H \cong G$ imply $K = 1$
(where $G = K \rtimes H$ means that  $H$ is a {\it retract\/} of $G$, i.
e. $G = K H$, $K \lhd G$, $K \cap H = 1$).
\end{defn}

\medskip
\begin{defn}
A group $G$ is {\it residually finite\/} if for every $g \in  G$, $g \neq 1$, there
exists a homomorphism $h$ from G onto a finite group $H < G$ such that $h(g) \neq 1$.
\end{defn}

\medskip
\begin{defn}
 A  module $M$ is called {\it Hopfian\/} if every
homomorphism $f: M \rightarrow M$ which is an epimorphism is an
isomorphism.
\end{defn}

\bigskip
\begin{rem}  Let $P$ be a polyhedron such that $\pi_1(P)$ is
weakly Hopfian and all the $\pi_i(X)$,
for $i = 2, \ldots, \dim P$, are Hopfian modules over
$Z \pi_1(P)$.
Then, by the Whitehead Theorem,
 every homotopy domination $d: P \rightarrow P$
 is a homotopy
equivalence.
\end{rem}

\bigskip
For the following theorem, see [K, Theorem 3, Theorem 5]:

\bigskip
{\sc Theorem.} {\it Let $P$ be a polyhedron  such that
the group $\pi_1(P)$ is polycyclic-by-finite. Then
every homotopy domination $d: P \rightarrow P$ is
a homotopy equivalence.}

\begin{flushright}
$\Box$
\end{flushright}

\begin{rem} In the proof of the above theorem in [K] we
applied the fact that if $G$ is policyclic-by-finite, then
every finitely generated module over $ZG$ is Hopfian
(for details and necessary references, see the
proofs in [K]).
\end{rem}

\medskip
Therefore it is worth to ask:

\begin{problem}
 Does there exist a polyhedron $P$
 with $\dim P = 2$ and soluble fundamental group
and a homotopy domination of $P$ over itself that is not a
homotopy equivalence?
\end{problem}

\medskip
%\begin{defn}
%(i) Recall that by a {\it cohomological dimension\/}
%of a group $G$ we mean  \em{cd}G
%= \sup \{$i s uch that $H^{i}(G,M)\neq 0$, for some
%$ZG$-module M\}. (ii) A {\it geometrical dimension\/},
%$\em{gd}G,\/$ of a group $G$ is a smallest dimension of a
%$CW$-complex $K(G,1)$  (see [Br], [BK]).
%\end{defn}

\begin{defn}
(i) Recall that by a {\it cohomological dimension\/}
of a group $G$ we mean  cd$G$
= $\sup$ \{$i$ such that $H^{i}(G,M)\neq 0$, for some
$ZG$-module $M$\}.
(ii) A {\it geometrical dimension\/},
gd$G$, of a group $G$ is a smallest dimension of a
$CW$-complex $K(G,1)$  (see [Br], [BK]).
\end{defn}

\medskip
From the results of this paper follows that for each
 polyhedron $P$ with  $\dim P = 2$ and so\-lu\-ble fundamental group  $G =
\pi_1(P)$ satisfying cd$G \leq 2$, every homotopy
domination of $P$ over itself is a homotopy equivalence.

\vskip 2mm
\bigskip
From now on $X \leq Y$ will denote that $X$ is homotopy
dominated by $Y$.

\vskip 2mm
\bigskip
\section {Main Theorems}

\bigskip
\begin{defn}
 A group presentation is said to be {\it
aspherical\/} if the standard $2$-dimensional $CW$-complex
associated with it is aspherical.
\end{defn}

\smallskip
\begin{thm}  %[Main Theorem]
Let $P$ be a polyhedron with $\dim P = 2$ such that
 $\pi_1(P)$ has a finite aspherical presentation and is weakly
Hopfian. Then every homotopy domination $d: P \rightarrow P$
is a homotopy equivalence.
\end{thm}

\begin{proof}
Suppose that there exists a homotopy domination $P \geq P$
that is not homotopy equivalence. Since $G = \pi_1(P)$ is
weakly Hopfian, this domination induces an isomorphism
$\pi_1(P) \rightarrow \pi_1(P)$.

\bigskip
The Whitehead Theorem on Trees [Wh] states that, if $P$
and $Q$ are two finite
$2$-dimensional $CW$-complexes with $\pi_1(P)\cong \pi_1(Q)$,
then there exist integers $m_P$ and $m_Q$ such that $$P \vee
\bigvee\limits_{m_P} S^2 \simeq Q \vee
\bigvee\limits_{m_Q} S^2.$$

\smallskip
 Since there exists a finite $2$-dimensional $CW$-complex
 $K = K(G,1)$, it follows that
 $$P \vee \bigvee\limits_{m_P} S^2 \simeq K \vee
\bigvee\limits_{m_K} S^2,$$ for some integers $m_P$ and $m_K$.  Hence we have an isomorphism of $ZG$-modules
  $$\pi_2(P) \oplus (Z G)^{(m_P)} \cong
(ZG)^{(m_K)}.$$

\smallskip
Indeed, it is known that for any $2$-dimensional polyhedron
$Q$, $\pi_2(Q
\vee S^2) \cong
\pi_2(Q) \oplus \pi_2(S^2)$ as $Z \pi_1(Q)$-modules.

\medskip
Suppose that $d: P \rightarrow P$ is a homotopy
domination but not a
homotopy equivalence. Then exists a nontrivial $ZG$-module
$N$ such
that $\pi_2(P) \oplus N \cong \pi_2(P)$.
Otherwise, d induces also an isomorphism
of $\pi_2$, and  is a homotopy equivalence, whis
is a contradiction.

\medskip
Therefore $(ZG)^{(m_K)} \oplus N \cong (ZG)^{(m_K)}$. Thus
$(ZG)^{(m_K)}$ is isomorphic to a proper direct factor of
itself.
But this is impossible. Indeed, for any group G, any finitely
generated free $ZG$-module $(ZG)^{(m)}$ (where $m$ is an
integer)
cannot be isomorphic to a proper direct factor of itself ---
 from the result of I. Kaplansky  [Ka, p.122]. \end{proof}

\medskip
\begin{rem} (i) Obviously, every Hopfian group is weakly Hopfian. (ii) It should be noted that there is not known at present any example of a finitely presented group that is not weakly Hopfian. (iii) On the other hand, weak hopficity was proven for some classes of finitely presented groups. For example, every nilpotent-by-nilpotent group is weakly hopfian. (iv) It is also known that any  finitely generated
abelian-by-nilpotent group is Hopfian (P. Hall).
\end{rem}

\bigskip
So let us formulate the following:

\begin{problem}
Does there exist a finitely presented
group which  is not weakly Hopfian?
\end{problem}

\begin{problem}
Does there exist a finitely presented
group which  has a finite aspherical presentation
and is not weakly
Hopfian?
\end{problem}

\smallskip
\begin{cor}  Let $P$ be a polyhedron with $\dim P = 2$
such that
 $\pi_1(P)$ has a~finite aspherical presentation and is Hopfian.
Then every homotopy domination $d: P \rightarrow P$ is a homotopy
equivalence.
\end{cor}

\medskip
From the above we obtain:

\begin{thm}
 {\it Let $P$ be a polyhedron with
$\dim P = 2$ such that $\pi_1(P)$ is a knot group.  Then every
homotopy domination $d: P \rightarrow P$ is a homotopy
equivalence.}
\end{thm}

\begin{proof}
Any knot group $G$ has a finite aspherical
presentation --- there exists
a finite $CW$-complex $K(G,1)$ of dimension $2$
(by the result of [P], see also [Br]).
It is also known that any knot group is residually
finite ([Th],
see, for example, [Kb]). Any finitely generated
residually finite group is Hopfian [Ma]. Thus the proof
is completed.
\end{proof}

\begin{rem}
A special case of 2-dimensional polyhedra with fundamental
groups isomorphic to the Trefoil knot group $T = <a,b
\mid \hspace{0.2cm} a^{2}=b^3>$ was considered by the
author in [K1].
\end{rem}

\begin{thm} Let $P$ be a polyhedron with $\dim P = 2$ such
that $\pi_1(P)$ is one-related, torsion-free and weakly Hopfian.
Then every homotopy domination $d: P \rightarrow P$ is a~homotopy
equivalence.
\end{thm}

\begin{proof}
Every one-related, torsion-free group has a finite aspherical
presentation. Precisely, a $CW$-complex naturally
corresponding to a given one-related presentation
(created by adding a single $2$-cell to a finite wedge
of circles corresponding to the generators) is aspherical
(compare [Ly], [Br]). So it follows from Theorem~
1.
\end{proof}

\begin{cor}
{\it Let $P$ be a polyhedron with $\dim P = 2$ such that
$\pi_1(P)$ is  one-related, torsion-free and Hopfian.
Then every homotopy domination $d: P \rightarrow
P$ is a homotopy equivalence.}
\end{cor}

\begin{rem}
Recall that  if a finite presentation of some group has
exactly one relator that is not a
proper power (i. e., a power of some element of this group),
then this group is torsion-free.
\end{rem}

\begin{rem}
(i) In many cases one-related groups are known to be Hopfian
(see, for example, [W], [W1], [ARV], [CL], [S]).
(ii) A recent results of M. Sapir and I. Spakulova
[SSp] show that  almost surely, a one-relator group with
at least 3 generators is residually finite, hence Hopfian.
\end{rem}

\begin{defn}[Hyperbolic Groups] Recall that a finitely
generated group is called {\it  hyperbolic\/} if its Calley
graph with respect to some finite generatic set is hyperbolic
[Gr].
\end{defn}

\begin{rem} (i) Almost every finitely presented group is
hyperbolic (see M. Gromov [Gr1]).
(ii) Examples of non-hyperbolic groups are $Z \times Z$
and just considered here knot groups.
(iii) All the hyperbolic groups without torsion have
finite Eilenberg-Mac Lane $CW$-complexes
(see, for example, [Kt]).
\end{rem}

\begin{thm}
 Let $P$ be a polyhedron with $\dim P
= 2$ such that  $\pi_1(P)$ is a hyperbolic group and has
a finite
aspherical presentation. Then every homotopy domination
$d: P
\rightarrow P$ is a homotopy equivalence.
\end{thm}

\begin{proof}  By the result of Z. Sela [Se], every
hyperbolic group is Hopfian.  Hence it follows
from Corollary 1.
\end{proof}

\medskip
In the sequel, we will use properties of Baumslag-Solitar groups
$BS(m,n)$ and consider soluble groups (and more general, elementary
amenable groups).

\begin{defn}[Baumslag-Solitar Groups]
For each pair of integers  $0 < m \leq |n|$,
$BS(m,n)= <a,b \mid \hspace{0.2cm} ab^m a^{-1} = b^{n}>$
\hskip 2mm (compare [BS]).
\end{defn}

\begin{rem}
It is known that
(i)  $BS(m,n)$  soluble (but non-nilpotent)
if and only if $m=1$. These groups are, in particular,
%abelian-by-cyclic. They are also 
metabelian, hence Hopfian.
(ii) in general, $BS(m,n)$ is Hopfian if and only if $m=1$ or
$m$ and $n$ have the
same prime divisors.  %For example $BS(2,3)$ is non-Hopfian.
\end{rem}

\medskip
\begin{defn}[Elementary Amenable Groups]
 {\it Elementary amenable groups\/} is the smallest class of groups that contains all abelian and all
finite groups, and is closed under extensions and directed unions
(see [KLL]).
\end{defn}

\begin{thm}
 {\it Let $P$ be a polyhedron with $\dim P = 2$ such that
 $G = \pi_1(P)$ is elementary amenable and $cdG = 2$. Then
every homotopy domination $d: P \rightarrow P$ is a homotopy
equivalence.}
\end{thm}

\begin{proof} Let  $G$ be a finitely generated elementary
amenable group  and $cdG = 2$. Then,  $G$ has
a presentation of the form
$G = <a,b \mid \hspace{0.2cm} aba^{-1}=b^m>$, for
some $m \in Z-\{0\}$ [KLL, Theorem 3], i. e. is a
Baumslag-Solitar group B(1,m). Then, there exists a
finite $CW$-complex K(G,1) of dimension 2
(as in the proof of
Theorem 2). Moreover,  every
Baumslag-Solitar group $B(1,m)$ is metabelian, hence
(in the case of
finitely generated groups) Hopfian (compare Remark 8).
So, we apply Corollary 1, and the proof is complete.
\end{proof}

\bigskip
As a corollary, we obtain the following:

\begin{thm} Let $P$ be a polyhedron with $\dim P = 2$
such that  $G = \pi_1(P)$ is soluble and $cdG = 2$.
Then every
homotopy domination $d: P \rightarrow P$ is a homotopy
equivalence.
\end{thm}

\begin{proof} This is a corollary to Theorem 5.
Any soluble group is  poly-abelian, hence elementary amenable.
\end{proof}

\medskip
Our results will be completed by the following.

\begin{defn} [Limit Groups] A finitely generated group $G$ is
a {\it limit group\/} if, for any subset $S \subset G$, there
exists a homomorphism $f: G \rightarrow F$ (where $F$ is
a free
group of finite rank) so that the restriction of
$f$ to $S$ is
injective.
\end{defn}

\begin{rem}
(i) Examples of limit groups include finite-rank free
abelian
groups. (ii) Limit groups are non-soluble except
of free abelian groups.
(iii) Note that all the limit groups have finite
Eilenberg-Mac Lane $CW$-complexes.
\end{rem}

\medskip
The following useful lemma can be drawn, for example,
from [AB]:

\begin{lem}
 {\it Any limit group is  Hopfian.\/}
\end{lem}

\begin{proof} Given a limit group $G$, any sequence
of epimorphisms $G = G_0 \rightarrow G_1 \rightarrow \cdots$
eventually consists of isomorphisms (see [AB, the proof of
Lemma A.1, p. 269]). Therefore, there is no an epimorphism
of G onto $G$ which is not an isomorphism
(take $G_i = G$, for all $i$, with the same
given epimorphism between $G_i$ and $G_{i+1}$).
Hence $G$ is Hopfian.
\end{proof}

\begin{thm}
 {\it Let $P$ be a polyhedron with
$\dim P = 2$ such that  $\pi_1(P)$ is a limit group and has
a finite aspherical presentation. Then every homotopy
domination $d: P \rightarrow P$ is a homotopy equivalence.\/}
\end{thm}

\begin{proof}
It follows from Corollary 1 and Lemma 1.
\end{proof}

\bigskip
\section{On different homotopy types of $ANR's$
do\-mi\-nating each other}

\medskip
In [B1, Ch.IX, (12.7)] K. Borsuk stated the following
question:

\begin{problem} Is it true that two $ANR's$, $P$ and
$Q$ homotopy dominating each other have the same
homotopy type?
\end{problem}

\medskip
By the previous results of this paper (Theorems 1-7,
Corollaries 1-2), we obtain:

\begin{cor}
 {\it Let $P, Q \in ANR$,  $\dim P = 2$,
 $\pi_1(P)$  has a finite aspherical presentation and is weakly
Hopfian. Then $P \geq Q$ and $Q \geq P$, implie that $P \simeq
Q$.}
\end{cor}

\begin{cor}
 {\it Let $P, Q \in ANR$, $\dim P =
2$,  $\pi_1(P)$
  has a finite presentation with one relation,
 is torsion-free and weakly Hopfian.
 Then $P \geq Q$ and $Q \geq P$, implie that
 $P \simeq Q$.}
\end{cor}

\begin{cor}
 {\it Let $P, Q \in ANR$,
$\dim P = 2$, and  $G = \pi_1(P)$ satisfies one of the
follo\-wing
conditions:

\begin{itemize}

\item[(i)]
 $G$  is soluble and $cdG \leq 2,$

\item[(ii)] $G$  is elementary amenable and $cdG \leq 2$,

\item[(iii)]
 $G$  is a limit group and has a finite aspherical
 presentation,

\item[(iv)]
 $G$  is a hyperbolic group and has a finite aspherical
 presentation,

\item[(v)]
 $G$  is a knot group.

\end{itemize}

 Then $P \geq Q$ and $Q \geq P$, implie that
 $P \simeq Q$.}
\end{cor}

\begin{flushright}
$\Box$
\end{flushright}

\medskip
One may ask the following questions (compare strongly
related Problem 2):

\begin{problem}
 Does there exist a finitely presented one-related
group $G$ and an $r$-homomorphism $r: G \rightarrow H$,
where $H
\cong G$, such that $r$ is not an isomorphism?
\end{problem}

\begin{problem}
 Do there exist  finitely presented
one-related, non-isomorphic groups $G$ and $H$ and
$r$-homomorphisms $h: G \rightarrow H$, and
$h: H \rightarrow G'$,
where $G' \cong G$?
\end{problem}

\bigskip
\section{On two quasi-homeomorphic $ANR's$ of different
homotopy types}

\medskip
The main problem we consider here is also related to the
other question published by K. Borsuk in [B2]. Let us recall
two definitions (see [MS] and  [B2], respectively).

\begin{defn}[S. Marde\v{s}i\'{c}, J. Segal]
Let $X$ and $Y$ be
compacta. $X$ is said to be $Y$-like, if for every
$\varepsilon
> 0$ there exists a continuous map
$f:X \xrightarrow{onto} Y$ such that, for  all $y
\in Y$,   diam$(f^{-1}(y)) <  \varepsilon$.
\end{defn}

\begin{defn}
[K. Kuratowski, S. Ulam] Two compacta
$X$ and $Y$ are {\it quasi-ho\-meo\-mor\-phic\/}
 if $X$ is $Y$-like and $Y$ is $X$-like.
 \end{defn}

\medskip
Borsuk found two quasi-homeomorphic compacta of the different
shapes and asked [B2, Problem (12.7), p. 231-233]:

\begin{problem}  Is it true that the shapes of
two quasi-homeomorphic ANRs are equal?
\end{problem}

\medskip
It is well-known that on ANR's shape and homotopy theory
coincide, thus in this problem shapes can be replaced
by homotopy types.

\smallskip
From Corollary 5 we can drawn:

\begin{cor}
{\it Let $P, Q \in ANR$,  $\dim P = 2$ and  $G = \pi_1(P)$
satisfies one of the follo\-wing conditions:

\begin{itemize}

\item[(i)]
 $G$  is soluble and $cdG \leq 2,$

\item[(ii)] $G$  is elementary amenable and $cdG \leq 2$,

\item[(iii)]
 $G$  is a limit group and has a finite aspherical
 presentation,

\item[(iv)]
 $G$  is a hyperbolic group and has a finite aspherical
 presentation,

\item[(v)]  $G$  is a knot group.

\end{itemize}
If $P$ and $Q$ are quasi-homeomorphic, then $P \simeq Q$.}

\end{cor}

\begin{proof}
It is known that, if $P$ and $Q$ are two $ANR$s, %$Y \in ANR$
and $Q$ is $P$-like, then $Q \leq P$ (compare [B2, (12.6)
p.~233]). Thus, if $P$ and $Q$ are two quasi-homeomorphic
$ANR$s,
then $P \leq Q$ and $Q \leq P$. We apply Corollary 5
and the proof is complete.
\end{proof}

\vskip 2mm
\bigskip
\section{Final Remarks}

\medskip
(1) Another classes of groups satisfying the
assumptions of the main theorem one can find
among subgroups of the Coxeter groups, small cancelations
groups, Fuchsian groups, and among others.

\end{document}